%% file: main.tex
\documentclass{amsart}
\author{Stuart Martin}
\address{Department of Pure Mathematics and Mathematical Statistics\\ University of Cambridge \\Cambridge CB3 0WA \\ UK}
\email{sm137@cam.ac.uk}

\include{header}

\newcommand{\pljwnQ}{\prescript{p}{\ell}{\JW}^{\Q(\delta)}_n}
\newcommand{\pljwnR}{\prescript{p}{\ell}{\JW}^{k}_n}
\newcommand{\pljwfnR}{\prescript{p}{\ell}{\JW}^{k}_{\f{n}}}
\newcommand{\pljwfnQ}{\prescript{p}{\ell}{\JW}^{\Q(\delta)}_{\f{n}}}

\newcommand{\pljwnmtQ}{\prescript{p}{\ell}{\JW}^{\Q(\delta)}_{n-t}}

\newcommand{\pljwnmtR}{\prescript{p}{\ell}{\JW}^{k}_{n-t}}

\graphicspath{{diagrams/}}

\title{$(\ell,p)$-Jones-Wenzl Idempotents}
\begin{document}

\begin{abstract}
  The Jones-Wenzl idempotents of the Temperley-Lieb algebra are celebrated elements defined over characteristic zero and for generic loop parameter.
  Given pointed field $(R, \delta)$, we extend the existing results of Burrull, Libedinsky and Sentinelli to determine a recursive form for the idempotents describing the projective cover of the trivial $\TL_n^R(\delta)$-module.
\end{abstract}

\maketitle

\section*{Introduction}\label{sec:intro}
\input{introduction}

\section{Notation and Conventions}\label{sec:notation}
\input{notation}


\section{Jones-Wenzl Idempotents}\label{sec:jones-wenzl}
\input{jwidempotents}

\section{\texorpdfstring{$p$}{p}-Jones-Wenzl Idempotents}\label{sec:p-jones-wenzl}
\input{pjwidempotents}

\section{Induction and Projective Covers of the Trivial Module}\label{sec:induction}
\input{induction}

\section{Main Result}\label{sec:results}
\input{results}

\section{Traces}\label{sec:traces}
\input{traces.tex}

\section*{Acknowledgements}
\noindent
The second author is supported by a DTP studentship from the Department of Pure Mathematics and Mathematical Sciences of the University of Cambridge funded by the Engineering and Physical Sciences Research Council.

\bibliographystyle{halpha}
\bibliography{all_cite}

\end{document}

%% file: header.tex
\usepackage{mathpazo}
\usepackage{hyperref}
\hypersetup{%
  colorlinks,               
  linkcolor={red!50!black}, 
  citecolor={blue!50!black},
  urlcolor={blue!80!black}  
}

\usepackage{amssymb}
\usepackage{amsmath}
\usepackage{amsfonts}
\usepackage{amsthm}

\usepackage{microtype}

\usepackage{relsize}

\usepackage[utf8]{inputenc}
\usepackage[final]{graphicx}
\usepackage{epstopdf}

\usepackage{indentfirst}
\usepackage{lipsum}

\usepackage{marvosym}
\usepackage{booktabs}
\usepackage{lscape}

\usepackage{mathtools}

\usepackage{array}

\usepackage{multirow}
\usepackage{multicol}

\usepackage[backgroundcolor=white,linecolor=lime,bordercolor=white,figcolor=white]{todonotes}
\makeatletter
\providecommand\@dotsep{5}\renewcommand{\listoftodos}[1][\@todonotes@todolistname]{%
  \@starttoc{tdo}{#1}}
\makeatother
\usepackage{enumerate}

\usepackage{xcolor}

\usepackage{iftex}

\usepackage[capitalise]{cleveref}
\usepackage{cancel}
\usepackage{diagbox}

\newcommand{\Hom}{\operatorname{Hom}}
\newcommand{\End}{\operatorname{End}}
\newcommand{\supp}{\operatorname{supp}}

\newcommand{\id}{\operatorname{id}}

\newcommand{\Z}{\mathbb{Z}}
\newcommand{\N}{\mathbb{N}}

\newcommand{\Q}{\mathbb{Q}}
\newcommand{\F}{\mathbb{F}}

\newcommand{\mathscr}{\mathcal}

\definecolor{rdylgn60}{HTML}{D73027}
\definecolor{rdylgn61}{HTML}{FC8D59}
\definecolor{rdylgn62}{HTML}{FEE08B}
\definecolor{rdylgn63}{HTML}{D9EF8B}
\definecolor{rdylgn64}{HTML}{91CF60}
\definecolor{rdylgn65}{HTML}{1A9850}
\definecolor{rdylgn66}{HTML}{006837}
\definecolor{ca1}{HTML}{72497F}
\definecolor{ca2}{HTML}{5C7B78}
\definecolor{cb1}{HTML}{0862A0}
\definecolor{cb2}{HTML}{CD9B3E}
\definecolor{cc1}{HTML}{C55824}
\definecolor{cc2}{HTML}{288338}
\definecolor{cd}{HTML}{CF838A}
\definecolor{ce}{HTML}{FFAF3A}
\definecolor{cf}{HTML}{AFBF5A}
\definecolor{cg}{HTML}{7ACF99}
\definecolor{ch}{HTML}{8AB1CF}
\definecolor{ci}{HTML}{3B58CF}
\definecolor{cj}{HTML}{736163}

\definecolor{lightgray}{HTML}{DDDDDD}

\makeatletter
\@ifclassloaded{beamer}{%
}{%
\usepackage[margin=1.5in]{geometry}
\makeatother
\theoremstyle{plain}
\newtheorem{theorem}{Theorem}[section]
\newtheorem{proposition}[theorem]{Proposition}
\newtheorem{definition}[theorem]{Definition}

\newtheorem{lemma}[theorem]{Lemma}

\newtheorem{corollary}[theorem]{Corollary}
}

\theoremstyle{remark}

\usepackage{genyoungtabtikz}

\numberwithin{equation}{section}

\usepackage{tikz}
\usetikzlibrary{shapes,arrows}

\usetikzlibrary{decorations.pathreplacing,decorations.pathmorphing}

\usepackage{tikz-cd}
\ifLuaTeX%
\usetikzlibrary{graphdrawing}
\usetikzlibrary{graphdrawing.layered}
\fi

\tikzset{%
  spot/.style={color=black, thin, dashed},
  rline/.style={color=green, line width=2pt},
  sline/.style={color=blue, line width=2pt},
  tline/.style={color=red, line width=2pt},
  uline/.style={color=green, line width=2pt},
  line/.style={color=#1, line width=2pt},
  line/.default=blue,
  rdot/.style={color=green, thin, fill},
  sdot/.style={color=blue, thin, fill},
  tdot/.style={color=red, thin, fill},
  udot/.style={color=green, thin, fill},
  dot/.style={color=#1, thin, fill},
  dot/.default=blue
}

\newcommand{\JW}{{\rm JW}}
\newcommand{\TL}{{\rm TL}}
\newcommand{\f}[1]{{\mathrm{f}[#1]}}

\newcommand{\TLcat}{{\mathscr{T\!\!L}}}

\newcommand{\gaussianquant}{\genfrac{[}{]}{0pt}{}}

\newcommand{\restr}{\mathord\downarrow}
\newcommand{\ind}{\mathord\uparrow}

\author[R. A. Spencer]{Robert A. Spencer}
\address{Department of Pure Mathematics and Mathematical Statistics\\ University of Cambridge \\Cambridge CB3 0WA \\ UK}
\email{ras228@cam.ac.uk}
\date{}

%% file: introduction.tex
The Temperley-Lieb algebra, $\TL_n$, defined over ring $R$ with distinguished element $\delta$ by the generators ${\{u_i\}}_{i=1}^{n-1}$ and relations
\begin{align}\label{eq:def-1}
  u_i ^2 &= \delta u_i\\\label{eq:def-2}
  u_i u_j &= u_j u_i & |i-j| \ge 2\\\label{eq:def-3}
  u_i u_{i\pm1}u_i &= u_i & 1 \le i \pm 1 < n,
\end{align}
has recently been recast into the limelight.

First studied over characteristic zero, as algebras of transfer operators in lattice models, these structures found extensive use in physics and, later, knot theory where Vaughan Jones famously used them to define the Jones polynomial invariant.

More recently, Temperley-Lieb algebras and their variants have become the subject of study by those seeking to understand Soergel bimodule theory~\cite{elias_2016}.
Here they are intricately linked to the categorification of the Hecke algebra of type $\tilde{A_1}$ by two-colour Soergel bimodules.
From this categorification arises interesting ``canonical'' bases of certain Hecke algebras, of which the Kazhdan-Lusztig basis is probably the most famed.

However, other interesting bases occur, particularly when the underlying ring has positive characteristic.
Work by Jensen and Williamson~\cite{jensen_williamson_2015} develops the so-called $p$-canonical basis or $p$-Kazhdan-Lusztig basis for crystallographic Coxeter types.
The underlying calculations in the two-colour case reduce to results in the representation theory of the Temperley-Lieb algebra.

In the language of Soergel bimodules, Jones-Wenzl idempotents describe the indecomposable objects.
Recently, Burrull, Libedinsky and Sentinelli~\cite{burrull_libedinsky_sentinelli_2019} determined the corresponding elements of $\TL_n$ defined over a field of characteristic $p \ge 2$ for the parameter $\delta = 2$.
The results of Erdmann and Henke~\cite{erdmann_henke_2002} implicitly describing the $p$-canonical bases are crucial.

\vspace{1em}
\noindent
Throughout these results, the role of the Temperley-Lieb algebra as the centralizer ring $\End_{U_q(\mathfrak{sl}_2)}(V^{\otimes n})$ has been key to understanding its modular representation theory~(see \cite{andersen_2019, andersen_catharina_tubbenhauer_2018}, and the reliance on~\cite{erdmann_henke_2002} in \cite{cox_graham_martin_2003,burrull_libedinsky_sentinelli_2019} for examples).
The explicit nature of the tilting theory of $U_q(\mathfrak{sl}_2)$ has underpinned most of the results.

However, the algebras themselves admit a pleasing diagrammatic presentation stemming from \cref{eq:def-1,eq:def-2,eq:def-3}.
Most of the theory of the characteristic zero case was determined purely ``combinatorially'' from this~\cite{westbury_1995, ridout_saint_aubin_2014}.
In~\cite{spencer_2020_modular} the second author re-derives many of the results known about the representation theory of $\TL_n$ over positive characteristic without recourse to tilting theory of $U_q(\mathfrak{sl}_2)$.

This paper builds on~\cite{spencer_2020_modular} and~\cite{burrull_libedinsky_sentinelli_2019} to construct $(\ell, p)$-Jones-Wenzl idempotents giving the indecomposable objects in the most general case of any field and any parameter.
This answers one of the questions in~\cite[1.3.5]{burrull_libedinsky_sentinelli_2019} by showing that the construction given extends quite simply to other (non-integral) realisations.
That is, if the reflection representation of $D_{2\ell}$ of the form
\begin{equation*}
  s \mapsto \begin{pmatrix}-1&\delta\\0&1\end{pmatrix}\quad\quad\quad\quad
  t \mapsto \begin{pmatrix}1&0\\\delta&-1\end{pmatrix}
\end{equation*}
for adjacent reflection generators $s$ and $t$ of $D_{2\ell}$ is chosen for the $\tilde A_1$ Hecke category, then the indecomposable objects are given by $(\ell, p)$-Jones-Wenzl idempotents.
If the two generators have differing off-diagonal values so the realisation is not symmetric, minor changes to the result occur with quantum numbers replaced by two-colored quantum numbers~\cite[\S A.1]{elias_2016}.
The exact modification is not complicated (it predominantly consists of in-place replacements), but the careful handling of even and odd cases is beyond the scope of this paper.

\vspace{2em}
\noindent
This paper is arranged as follows.
We briefly recall some of the needed concepts and define our notation in \Cref{sec:notation}.
In \Cref{sec:jones-wenzl} we recall the known theory of Jones-Wenzl idempotents over characteristic zero and make some observations in positive characteristic.
In \Cref{sec:p-jones-wenzl} we recount the construction of $p$-Jones-Wenzl elements due to Burrull, Libedinsky and Sentinelli~\cite{burrull_libedinsky_sentinelli_2019}, with a slight modification to allow for generic parameter.
\Cref{sec:induction} collects the properties of projective modules for $\TL_n$ that will be needed in the result and we prove our main theorem in \cref{sec:results}.
Finally, \cref{sec:traces} examines the action of the Markov trace on our new elements.

%% file: notation.tex
We will use the notation and conventions for $\TL^R_n(\delta)$ set out in~\cite{spencer_2020_modular}.
Note that in this formulation, closed loops resolve to a factor of $\delta$ as opposed to $-\delta$ as is often found in the literature.
Throughout, we will omit the $\delta$ from our notation for the algebra $\TL_n^R$, as every ring $R$ discussed will be naturally unambiguously pointed.

The Temperley-Lieb category $\TLcat$ is simply the linear category on $\N$ with morphism spaces spanned by Temperley-Lieb diagrams.
Objects in this category will be written as $\underline{n}$.

Throughout we distinguish the \emph{diagram basis} for $\TL_n$.
This is a basis given by all $(n,n)$ Temperley-Lieb diagrams.
Since multiplication of elements in this basis has image in its $\Z[\delta]$-span, this will allow us to move between base rings.

The Temperley-Lieb algebras are all equipped with an anti-automorphism, denoted $\iota$, which acts on diagrams by vertical reflection.
This linear map satisfies $\iota(xy) = \iota(y)\iota(x)$ and $\iota(\iota x) = x$ for each pair of elements $x, y \in \TL_n$.
This is known as the \emph{cellular involution}.

Each Temperley-Lieb algebra admits a \emph{trivial module}.
This is a one-dimensional module on which every diagram except the identity acts as zero.
The identity diagram acts as one.
Equivalently, if $I_n$ is the maximal ideal generated by all non-identity diagrams, then the trivial module is $\TL_n / I_n$.

Of critical importance will be the \emph{standard modules} of $\TL_n$ also known as \emph{cell modules}, denoted $S(n,m)$ for $m \le n$ of the same parity.
These have basis given by diagrams from $n$ to $m$ with the natural left action of $\TL_n$, modulo diagrams with fewer than $m$ through-strands.
The trivial module is thus $S(n,n)$.

\vspace{1em}
\noindent
The \emph{quantum numbers} $[n]$ are polynomials in $\Z[\delta]$ given by $[0] = 0$, $[1] = 1$ and $[n+1] = \delta[n] - [n-1]$.
We will often specialise these to a given ring, if $\delta$ is understood, by considering their image under the natural ring homomorphism.
If $\delta \in R$, we say that $\delta$ \emph{satisfies} $[n]$ if this image of $[n]$ is zero in $R$.

\vspace{1em}
\noindent
If $\ell$ and $p$ are given, then the \emph{$(\ell, p)$-digits} or the \emph{$(\ell,p)$-expansion} of the natural number $n$ are the numbers $n_i$ such that $n = \sum_{i = 0}^b n_i p^{(i)}$ where
\begin{equation*}
  p^{(i)} = \begin{cases}1 & i = 0 \\p^{i-1}\ell & i> 0\end{cases}
\end{equation*}
and $0\le n_i < p^{(i+1)}$.

%% file: jwidempotents.tex
Here we discuss the Jones-Wenzl idempotents, their construction, and when they exist.

\subsection{Semi-simple case}
The Jones-Wenzl idempotent, denoted $\JW_n$, is a celebrated element of $\TL_n^{\Q(\delta)}$ (where $\delta$ is indeterminate).
It is the unique idempotent $e$ such that $\TL_n \cdot e$ is isomorphic to the trivial module.
As such it satisfies the relations
\begin{equation}\label{eq:left_jw_def}
  u_i \cdot \JW_n = 0 \quad \forall\; 1\le i < n.
\end{equation}
It is clear that since $\JW_n$ is idempotent, the coefficient of the identity diagram $\underline{n} \to \underline{n}$ is 1.

\begin{lemma}\label{lem:defining_jw}
  Let $R$ be any pointed ring.
  Suppose $e$ is an idempotent of $\TL_n^R$ such that \cref{eq:left_jw_def} holds.  Then $e$ is invariant under the cellular involution $\iota$, and $e \cdot u_i = 0$ for all $1\le i < n$.
  Thus $e$ is the unique idempotent of $\TL_n^R$ satisfying \cref{eq:left_jw_def}.
\end{lemma}
\begin{proof}
  Write $e = \id_n + f$.
  Thus $f$ lies in the ideal $I_n$, and in particular lies in the span of diagrams factoring through some $u_i$.
  Hence $fe = 0$.
  However $\iota(f)$ is also in $I_n$ so $(\iota f) e = 0$.
  Thus
  \begin{equation*}
    e = (\id_n + \iota f)\cdot e = (\iota e) e ,
  \end{equation*}
  which is fixed under $\iota$.
  Thus $0 = \iota(u_i \cdot e) = \iota e \cdot \iota u_i = e \cdot u_i$.

  Finally, if $e_1$ and $e_2$ are two such idempotents, then both have unit coefficient for the identity diagram and so $e_1 = e_1 e_2 = e_2$.
\end{proof}

We now argue that idempotents satisfying \cref{eq:left_jw_def} exist.
Indeed, the algebra $\TL_n^{\Q(\delta)}$ is semi-simple so the trivial module is projective.
Thus the idempotent $e$ such that $\TL_n^{\Q(\delta)}\cdot e$ is isomorphic to the trivial module suffices.

On the other hand, there is also a recursive formulation that explicitly constructs the idempotents:
\begin{lemma}\label{lem:recurse_jw_def}
  Note that $\JW_1 = \id_1$. 
  For $n>1$, 
  \begin{equation}\label{eq:jw_recurse}
    \JW_n = \JW_{n-1}\otimes\id_1 -\frac{[n-1]}{[n]}(\JW_{n-1}\otimes \id_1)\circ u_{n-1} \circ (\JW_{n-1}\otimes \id_1).
  \end{equation}
  In diagrams,
  \begin{equation}
    \vcenter{\hbox{
    \input{diagrams/jwrecurrance}
    }}
    =
    \vcenter{\hbox{
    \input{diagrams/jwrecurrance2}
    }}
    - \frac{[n-1]}{[n]}
    \vcenter{\hbox{
    \input{diagrams/jwrecurrance3}
}}.
  \end{equation}
\end{lemma}
For example,
\begin{equation}\label{eq:jw3}
  \vcenter{\hbox{
    \input{diagrams/jw3}
  }}
   = 
  \vcenter{\hbox{
    \input{diagrams/id3}
  }}
  -
  \frac{[2]}{[3]}
  \left(
  \vcenter{\hbox{
    \input{diagrams/jw3full1}
  }}
  +
  \vcenter{\hbox{
    \input{diagrams/jw3full2}
  }}
  \right)
  + \frac{1}{[3]}
  \left(
  \vcenter{\hbox{
    \input{diagrams/jw3full3}
  }}
  +
  \vcenter{\hbox{
    \input{diagrams/jw3full4}
  }}
  \right).
\end{equation}
An equivalent formulation is given by Morrison as follows:
\begin{lemma}\cite[4.1]{morrison_2014}
  Suppose $D$ is a diagram $\underline{n+1} \to \underline{n+1}$.  Let $\hat{D}$ be the diagram $\underline{n+2} \to \underline{n}$ formed by folding across the lowest target site of $D$.  Let $\{i\}$ be the set of positions of simple caps in $\hat{D}$ and $D_i\in \TL_n$ the diagrams obtained by removing those caps.  Then
  \begin{equation}
    \underset{\in \JW_{n+1}}{\textnormal{coeff}}(D) = \sum_{\{i\}} \frac{[i]}{[n+1]} \,\underset{\in \JW_n}{\textnormal{coeff}}\,(D_i)
  \end{equation}
\end{lemma}

There are other recursive relations both for the morphisms and their coefficients in various bases~\cite{sentinelli_2019} and expansions involving more or fewer copies of $\JW_n$ in the second term (so-called ``single'' and ``triple'' clasp formulae).
The authors are not aware of any formula for the coefficients of diagrams in the Jones-Wenzl idempotents in the semi-simple case that is not inherently recurrent.

\subsection{Characteristic Zero}
If we now specialise to a characteristic zero pointed field $(k, \overline\delta)$ where $\overline\delta$ satisfies  $[\ell]$ but no $[m]$ for $0 < m < \ell$, the Temperley-Lieb algebras are no longer semi-simple for all $n$.
In this case the trivial module is not, in general, projective, which means that the Jones-Wenzl elements ``do not exist''.

To be precise, let $m(\delta) \in \Z[\delta]$ be the minimal polynomial of $\overline\delta$ over the integers and $\mathfrak{p}$ the prime ideal of $\Q[\delta]$ generated by $m(\delta)$.
The element $\JW_n$ lies in $\TL_n^{\Q(\delta)}$.
We construct both the ``integer form'' of $\TL_n$ over $\Q[\delta]_\mathfrak{p}$ and the algebra of interest which is defined over the characteristic zero ``target'' field $\Q[\delta]_\mathfrak{p} / \mathfrak{p}_\mathfrak{p} \subseteq k$.
We can summarise these rings as:
\begin{center}
  \begin{tikzcd}
    \Q[\delta]_{\mathfrak{p}} \arrow[r,hook,"i"] \arrow[d,twoheadrightarrow]& \Q(\delta)\\
    \Q[\delta]_{\mathfrak{p}} / \mathfrak{p}_\mathfrak{p} \arrow[r,hook]&k
  \end{tikzcd}
\end{center}
If $n \ge \ell$ and $n \not \equiv_\ell -1$ then in the basis of diagrams, the Jones-Wenzl idempotent in $\TL_n^{\Q(\delta)}$ does not lie in the image of $i$ and thus cannot descend to an element of $\TL^k_n$.
Indeed if it did, then the trivial module would be projective which is not the case~\cite[8.1]{ridout_saint_aubin_2014}.

We can rephrase as follows.
If $n\ge \ell$ and $n \not \equiv_\ell -1$, then the Jones-Wenzl idempotent $\JW_n$ written in terms of diagrams cannot be put over a denominator not divisible by $m(\delta)$.

For example, observe that \cref{eq:jw3} makes no sense if $\overline{\delta} = \pm 1$ so that $[3] = 0$.
However, if $n \equiv_3 -1$ then indeed the equation makes sense.  As an example, set $\ell = 2$ so $\delta = 0$ and then \cref{eq:jw3} simplifies to the sum of three diagrams.

\vspace{1em}
If we have the particular case $n = \ell -1$ then Graham and Lehrer found an elegant ``non-recursive'' form for the coefficients of the diagrams in $\JW_n$.
Recall from~\cite{spencer_2020_modular} the definition of the ``candidate morphisms''
\begin{equation}\label{eq:hoop_element}
  v_{r,s} = \sum_{x} h_{F(x)} x
\end{equation}
and the subsequent proposition
\begin{proposition}\label{prop:jwform}\cite[3.6]{graham_lehrer_1998}
  If $s < r < s + 2\ell$ and $s + r \equiv_{\ell} -2$ then the map $S(n,r) \to S(n,s)$ given by $x \mapsto x \circ v_{r,s}$ is a morphism of $\TL_n$ modules for every $n$.
\end{proposition}

The immediate corollary (from setting $s = 0$ and $r = 2\ell-2$) is that the Jones-Wenzl idempotent $\JW_{\ell-1}$ exists over $k$ and it's diagram coefficients can be found by ``rotating the target sites'' and then computing the hook-formula $h_{F(x)}$.


\subsection{Positive Characteristic}
Let us now focus on the positive characteristic case.
As before, if no quantum number vanishes, we are in a semi-simple case and \cref{eq:jw_recurse} gives us all $\JW_n$.
Thus assume that we are working over a pointed field $(k, \overline{\delta})$ of characteristic $p$ and that $\ell$ is the least non-negative such that $[\ell]$ is satisfied by $\bar\delta \in k$.
Thus we say that we are under $(\ell, p)$-torsion.

Let $\overline{m}(\delta) \in \F_p[\delta]$ be the minimal polynomial satisfied by $\overline{\delta}$ and let $m(\delta)$ be a preimage in $\Z[\delta]$.
Then $\mathfrak{m} = (p, m(\delta))$ is a maximal ideal in $\Z[\delta]$.
Consider $S=\Z[\delta]_\mathfrak{m}$, a local noetherian domain with maximal ideal $\mathfrak{m}S$.
Now $m(x) \not\in (p)$ so, $(0) \subset (p)\subset \mathfrak{m}$ strictly and $S$ is regular of Krull dimension 2.
As such, its completion with respect to $\mathfrak{m}$, which we will call $R$ is regular and hence a domain.
Set $F$ to be the field of fractions of $R$.
This is a characteristic zero field containing $\Q(\delta)$.
\begin{center}
  \begin{tikzcd}
    R=\widehat{\Z[\delta]_\mathfrak{m}} \arrow[r,hook] \arrow[d,twoheadrightarrow]& F& \Q(\delta)\arrow[l,hook]\\
    R/\mathfrak{m}R \arrow[r,hook]&k
  \end{tikzcd}
\end{center}

What we have now is a ``$(\ell,p)$-modular system'' in that we have a triple $(F,R,k)$ such that $F$ is a characteristic zero field, which is the field of fractions of $R$, a complete local domain with residue field $k$.
Thus, any idempotent in an algebra defined over $k$ can be raised to an idempotent over $R$ and then injected in to one defined over $F$.

We would like to know if this is reversible.
That is, given an idempotent defined over $\Q(\delta)$, we consider it as an idempotent over $F$ and ask if it lies in the algebra over $R$.
If so, we may reduce it modulo the maximal ideal to find an idempotent over $k$.

In plain terms, we wish to know if the coefficients of the diagrams in $\JW_n$ can be written without denominators divisible by $p$ or $m(\delta)$.
If so, the idempotent ``exists'' in our field of positive characteristic.

Should an element $e$ satisfying \cref{eq:left_jw_def} exist over $k$, it is clear that $\TL_n^k \cdot e$ is a trivial module and so the trivial module is projective.
Thus we can raise the idempotent $e$ to an element of $\TL_n^{R}$ where action by $u_i$ sends the element to something divisible by $\mathfrak{m}^r$ for every $r$ and hence equal to zero.
That is to say, it lifts to $\JW_n$.

Thus we may consider an alternative defining property of the Jones-Wenzl element that it is the idempotent that generates the trivial module's cover, which is equal to the trivial module, whenever that is the case.

As such, the results of~\cite{spencer_2020_modular} (in particular Theorem 3.4 combined with 8.3) can be interpreted as follows:
\begin{theorem}\label{thm:when_jw}
  The Jones-Wenzl idempotent $\JW_n$ in $\TL_n^{\Q(\delta)}$ descends to an element of $\TL_n^k$ iff $n<\ell$ or $n < \ell p$ and $n\equiv_\ell -1$ or $n = a\ell p^{k}-1$ for some $k\ge 1$ and $1\le a < p$.
\end{theorem}
An equivalent statement is
\begin{corollary}\label{thm:when_jw_inv}
  The Jones-Wenzl idempotent $\JW_n$ descends to an element of  $\TL^k_n(\delta)$ iff the quantum binomials $\gaussianquant n r$ are invertible in $k$ for all $0\le r\le n$.
\end{corollary}
This has been shown by Webster in the appendix of~\cite{elias_libedinsky_2017}.
However the proof there, and the alternative provided in~\cite{webster_2014} both construct linear maps of representations of $U_q(\mathfrak{sl}_2)$ which are shown to be morphisms (and hence elements of $\TL_n$).
This is the only proof not to use the Schur-Weyl duality of which the authors are aware.

Regrettably, the authors are not aware of any formula in the style of \cref{eq:hoop_element} giving the form of these idempotents in generality.
However, we are able to illuminate one further case constructively.

Given $\ell$ and $p$, we write $p^{(i)} = \ell p^{i-1}$ for all $i>0$ and $p^{(0)} = 1$.
We will explicitly construct $\JW_{2p^{(r)}-1}$.
Key in the below is that the construction of $\JW_{p^{(r)}-1}$ as an ``unfolding'' of a trivial submodule of $S(2p^{(r)}-2,0)$ shows that the first trace vanishes (it corresponds by the action of $u_{p^{(r)}-1}$ on that module).
\begin{proposition}\label{prop:construct_2}
  The element $\JW_{2p^{(r)} -1}$ exists in $\TL^k_{2p^{(r)}-1}$ for each $r \ge 1$.
\end{proposition}
\begin{proof}
  We construct an element of $\TL_{2p^{(r)}-1}$ for which the identity diagram appears with coefficient 1 and which is killed by all cups and caps.
  It is clear that this element then satisfies \cref{eq:left_jw_def}.

  To do this, consider the ``turn up'' operator which takes a morphism $\underline{n} \to \underline{m}$ to a morphism $\underline{n+1}\to\underline{m-1}$ by ``turning the top strand back''.  Formally, this is the map
\begin{equation}
  u_{m,1} : x \mapsto (\id \otimes x)\circ(\cap\otimes \id_{m-1}).
\end{equation}
  We can iterate this to turn multiple strands:
  set $u_{m,k} = u_{m-k+1,1}\circ u_{m, k-1}$.

  Now, let $J_i = u_i(\JW_{p^{(r)}-1})$ and $J_{-i} = \iota\left(J_i \right)$ for $0\le i \le p^{(r)}-1$.
  Thus $J_i$ is a morphism $\underline{p^{(r)}-1 + i} \to \underline{p^{(r)} -1 - i}$.
  In diagrams,
  \begin{equation}
    J_i =
    \vcenter{\hbox{
      \input{diagrams/jwprt}
    }}
    \quad\quad\quad 0 \le i \le p^{(r)} -1
  \end{equation}
  and
  \begin{equation}
    J_{-i} =
    \vcenter{\hbox{
      \input{diagrams/jwprti}
    }}
    \quad\quad\quad 0 \le i \le p^{(r)} -1
  \end{equation}
  Recall that $\TLcat$ is equipped with an automorphism which flips all diagrams vertically.
  Let $K_i$ be the image of $J_i$ under this morphism.

  The elements $J_i$ and $K_i$ have been chosen such that composition with any cup or cap results in the zero morphism.
  Key to this observation is that the idempotent $\JW_{p^{(r)}-1}$ vanishes under the trace defined in \cref{sec:traces}.

  Now we may define
  \begin{equation}\label{eq:jw2lpk1def}
    e = \sum_{i = -{p^{(r)}-1}}^{p^{(r)}-1} {(-1)}^i \iota (K_i) \otimes \id_1 \otimes J_i.
  \end{equation}
  Diagrammatically, a typical term in the sum (where $i \ge 0$) looks like
  \begin{equation}
    (-1)^{i}
    \vcenter{\hbox{
      \input{diagrams/jw_2lpk_1_1}
    }}
  \end{equation}
  The sign coefficient of the summand has been chosen such that the term for which $i = 0$ (i.e. the unique summand containing the identity diagram) has coefficient 1.
  It will thus suffice to show that $e$ is killed by the (left) action of $u_j$ for each $1 \le j < p^{(r)}-1$.

  We know that $\iota(K_i)$ and $J_i$ are killed by all cups on the left and so the only terms in \cref{eq:jw2lpk1def} that do not vanish are those for which $p^{(r)}-i=j$ or $p^{(r)}-i = j+1$.

  These terms are identical up to sign (in which they differ) and so cancel.  They are both given (up to sign) by the diagram in \cref{eq:jw2lpk1kill} and can be written as ${(-1)}^i \iota(K_j) \otimes \cap \otimes J_j$.
  \begin{equation}\label{eq:jw2lpk1kill}
    \vcenter{\hbox{
      \input{diagrams/jw_2lpk_1_2}
    }}
  \end{equation}
  Since all the terms in the summand are sent to zero by all $u_j$, and the coefficient of the identity diagram is an idempotent, we can invoke \cref{lem:defining_jw} to show that this is indeed $\JW_{2p^{(r)}-1}$.
\end{proof}

Readers familiar with the Dihedral Cathedral~\cite{elias_2016} may be familiar with the ``circular'' Jones-Wenzl notation.
This stems from the observation that rotation of the Jones-Wenzl element, $\JW_{p^{(r)}-1}$, by a single strand leaves it invariant.
In the general case where $\JW_n$ is defined, this is no longer so.

%% file: diagrams/jwrecurrance.tex
\begin{tikzpicture}[scale=0.5]
  \draw (0,0) rectangle (2.5,2.5);
  \node at (1.25, 1.25) {$\JW_{n+1}$};
  \draw[very thick] (-0.5, 0.2) -- (0, 0.2);
  \draw[very thick] (-0.5, 2.3) -- (0, 2.3);
  \node at (-0.25, 1.5) {$\vdots$};
  \draw[very thick] (2.5, 0.2) -- (3, 0.2);
  \draw[very thick] (2.5, 2.3) -- (3, 2.3);
  \node at (2.75, 1.5) {$\vdots$};
\end{tikzpicture}

%% file: diagrams/jwrecurrance2.tex
\begin{tikzpicture}[scale=0.5]
  \draw (0,0) rectangle (2,2);
  \node at (1, 1) {$\JW_{n}$};
  \draw[very thick] (-0.5, 0.2) -- (0, 0.2);
  \draw[very thick] (-0.5, 1.7) -- (0, 1.7);
  \node at (-0.25, 1.2) {$\vdots$};
  \draw[very thick] (2.5, 0.2) -- (2, 0.2);
  \draw[very thick] (2.5, 1.7) -- (2, 1.7);
  \node at (2.25, 1.2) {$\vdots$};
  \draw[very thick] (-0.5, -0.5) -- (2.5, -0.5);
\end{tikzpicture}

%% file: diagrams/jwrecurrance3.tex
\begin{tikzpicture}[scale=0.5]
  \draw (-2.75,0) rectangle (-0.75,2);
  \node at (-1.75, 1) {$\JW_n$};
  \draw (0.75,0) rectangle (2.75,2);
  \node at (1.75, 1) {$\JW_n$};
  \draw[very thick] (-3.25, -0.5) -- (-0.75, -0.5);
  \draw[very thick] (3.25, -0.5) -- (0.75, -0.5);
  \draw[very thick] (-3.25, 0.25) -- (-2.75, 0.25);
  \draw[very thick] (-3.25, 1.75) -- (-2.75, 1.75);
  \draw[very thick] (3.25, 0.25) -- (2.75, 0.25);
  \draw[very thick] (3.25, 1.75) -- (2.75, 1.75);
  \draw[very thick] (-0.75, 0.5) -- (0.75, 0.5);
  \draw[very thick] (-0.75, 1.75) -- (0.75, 1.75);
  \node at (-3.05, 1.25) {$\vdots$};
  \node at (0, 1.3) {$\vdots$};
  \node at (3.05, 1.25) {$\vdots$};
  \draw [very thick] (-0.75, -0.5) arc (-90:90:0.35);
  \draw [very thick] (0.75, -0.5) arc (270:90:0.35);
\end{tikzpicture}

%% file: diagrams/jw3.tex
\begin{tikzpicture}[scale=0.5]
  \draw (0,0) rectangle (2,2);
  \node at (1, 1) {$\JW_3$};
  \draw[very thick] (-0.5, 1.6) -- (0, 1.6);
  \draw[very thick] (-0.5, 1.0) -- (0, 1.0);
  \draw[very thick] (-0.5, 0.4) -- (0, 0.4);
  \draw[very thick] (2.5, 1.6) -- (2, 1.6);
  \draw[very thick] (2.5, 1.0) -- (2, 1.0);
  \draw[very thick] (2.5, 0.4) -- (2, 0.4);
\end{tikzpicture}

%% file: diagrams/id3.tex
\begin{tikzpicture}[scale=0.5]
  \draw[very thick] (1, 1.6) -- (0, 1.6);
  \draw[very thick] (1, 1.0) -- (0, 1.0);
  \draw[very thick] (1, 0.4) -- (0, 0.4);
\end{tikzpicture}

%% file: diagrams/jw3full1.tex
\begin{tikzpicture}[scale=0.5]
  \draw[very thick] (1, 1.6) -- (0, 1.6);
  \draw[very thick] (1, 0.4) arc (270:90:0.3);
  \draw[very thick] (0, 1.0) arc (90:-90:0.3);
\end{tikzpicture}

%% file: diagrams/jw3full2.tex
\begin{tikzpicture}[scale=0.5]
  \draw[very thick] (1, 0.4) -- (0, 0.4);
  \draw[very thick] (1, 1.0) arc (270:90:0.3);
  \draw[very thick] (0, 1.6) arc (90:-90:0.3);
\end{tikzpicture}

%% file: diagrams/jw3full3.tex
\begin{tikzpicture}[scale=0.5]
  \draw[very thick] (0, 0.4) to[out=0,in=180] (1.0, 1.6);
  \draw[very thick] (0, 1.6) arc (90:-90:0.3);
  \draw[very thick] (1.0, 0.4) arc (270:90:0.3);
\end{tikzpicture}

%% file: diagrams/jw3full4.tex
\begin{tikzpicture}[scale=0.5]
  \draw[very thick] (0, 1.6) to[out=0,in=180] (1.0, 0.4);
  \draw[very thick] (0, 1.0) arc (90:-90:0.3);
  \draw[very thick] (1.0, 1.0) arc (270:90:0.3);
\end{tikzpicture}

%% file: diagrams/jwprt.tex
\begin{tikzpicture}[scale=0.5]
  \draw (5,1) rectangle (1,5);
  \node at (3,3) {$\JW_{p^{(r)}-1}$};
  \draw[very thick] (1,1.2) -- (0,1.2);
  \node at (0.5,3.2) {$\vdots$};
  \draw[very thick] (1,1.6) -- (0,1.6);
  \draw[very thick] (1,4.8) -- (0,4.8);
  \draw [decorate,decoration={brace,amplitude=5}] (0,1) -- (0,5);
  \node[rotate=-90] at (-0.9,3) {\footnotesize$p^{(r)}-1$};

  \draw[very thick] (5,4.8) arc (-90:90:0.2);
  \draw[very thick] (5,3.6) arc (-90:90:1.4);
  \draw[very thick] (5,3.2) arc (-90:90:1.8);
  \draw[very thick] (5,5.2) -- (0,5.2);
  \draw[very thick] (5,6.4) -- (0,6.4);
  \draw[very thick] (5,6.8) -- (0,6.8);
  \draw [decorate,decoration={brace,amplitude=5}] (0, 5) -- (0,7);
  \node[rotate=-90] at (-.9, 6.0) {\footnotesize$i$};
  \node at (3,6) {\footnotesize$\vdots$};

  \draw[very thick] (5,1.2) -- (8.5,1.2);
  \draw[very thick] (5,1.6) -- (8.5,1.6);
  \draw[very thick] (5,2.8) -- (8.5,2.8);
  \draw [decorate,decoration={brace,amplitude=5}] (8.5,3) -- (8.5,1);
  \node[rotate=-90] at (9.3,2) {\footnotesize$p^{(r)}-1-i$};
  \node at (6.7,2.4) {\footnotesize$\vdots$};
\end{tikzpicture}

%% file: diagrams/jwprti.tex
\begin{tikzpicture}[scale=0.5]
  \draw (-5,1) rectangle (-1,5);
  \node at (-3,3) {$\JW_{p^{(r)}-1}$};
  \draw[very thick] (-1,1.2) -- (-0,1.2);
  \node at (-0.5,3.2) {$\vdots$};
  \draw[very thick] (-1,1.6) -- (-0,1.6);
  \draw[very thick] (-1,4.8) -- (-0,4.8);
  \draw [decorate,decoration={brace,amplitude=5}] (-0, 5) -- (-0,1);
  \node[rotate=-90] at (--0.9,3) {\footnotesize$p^{(r)}-1$};

  \draw[very thick] (-5,4.8) arc (270:90:0.2);
  \draw[very thick] (-5,3.6) arc (270:90:1.4);
  \draw[very thick] (-5,3.2) arc (270:90:1.8);
  \draw[very thick] (-5,5.2) -- (-0,5.2);
  \draw[very thick] (-5,6.4) -- (-0,6.4);
  \draw[very thick] (-5,6.8) -- (-0,6.8);
  \draw [decorate,decoration={brace,amplitude=5}] (-0, 7) -- (-0, 5);
  \node[rotate=-90] at (--.9, 6.0) {\footnotesize$i$};
  \node at (-3,6) {\footnotesize$\vdots$};

  \draw[very thick] (-5,1.2) -- (-8.5,1.2);
  \draw[very thick] (-5,1.6) -- (-8.5,1.6);
  \draw[very thick] (-5,2.8) -- (-8.5,2.8);
  \draw [decorate,decoration={brace,amplitude=5}] (-8.5, 1) -- (-8.5,3);
  \node[rotate=-90] at (-9.3,2) {\footnotesize$p^{(r)}-1-i$};
  \node at (-6.7,2.4) {\footnotesize$\vdots$};
\end{tikzpicture}

%% file: diagrams/jw_2lpk_1_1.tex
\begin{tikzpicture}[scale=0.5]
  \draw (5,1) rectangle (1,5);
  \node at (3,3) {$\JW_{p^{(r)}-1}$};
  \draw[very thick] (1,1.2) -- (0,1.2);
  \node at (0.5,3.2) {$\vdots$};
  \draw[very thick] (1,1.6) -- (0,1.6);
  \draw[very thick] (1,4.8) -- (0,4.8);
  \draw [decorate,decoration={brace,amplitude=5}] (0,1) -- (0,5);
  \node[rotate=-90] at (-0.9,3) {\footnotesize$p^{(r)}-1$};

  \draw[very thick] (5,4.8) arc (-90:90:0.2);
  \draw[very thick] (5,3.6) arc (-90:90:1.4);
  \draw[very thick] (5,3.2) arc (-90:90:1.8);
  \draw[very thick] (5,5.2) -- (0,5.2);
  \draw[very thick] (5,6.4) -- (0,6.4);
  \draw[very thick] (5,6.8) -- (0,6.8);
  \draw [decorate,decoration={brace,amplitude=5}] (0, 5) -- (0,7);
  \node[rotate=-90] at (-.9, 6.0) {\footnotesize$i$};
  \node at (3,6) {\footnotesize$\vdots$};

  \draw[very thick] (5,1.2) -- (8.5,1.2);
  \draw[very thick] (5,1.6) -- (8.5,1.6);
  \draw[very thick] (5,2.8) -- (8.5,2.8);
  \draw [decorate,decoration={brace,amplitude=5}] (8.5,3) -- (8.5,1);
  \node[rotate=-90] at (9.3,2) {\footnotesize$p^{(r)} -i-1$};
  \node at (6.7,2.4) {\footnotesize$\vdots$};

  \draw[very thick] (3,7.2) -- (5.5,7.2);
  \draw[very thick] (5.5, 7.2) to[out=0, in=180] (8.5, 5.2);
  \draw[very thick] (0,9.2) to[out=0, in=180] (3,7.2);

  \draw (3.5, 13.4) rectangle (7.5,9.4);
  \node at (5.5,11.4) {$\JW_{p^{(r)}-1}$};
  \draw[very thick] (7.5,13.2) -- (8.5,13.2);
  \draw[very thick] (7.5,12.8) -- (8.5,12.8);
  \node at (8, 11.4) {$\vdots$};
  \draw[very thick] (7.5,9.6) -- (8.5,9.6);
  \draw [decorate,decoration={brace,amplitude=5}] (8.5,13.4) -- (8.5,9.4);
  \node[rotate=-90] at (9.3,11.4) {\footnotesize$p^{(r)}-1$};

  \draw[very thick] (3.5,9.6) arc (90:270:0.2);
  \draw[very thick] (3.5,10.0) arc (90:270:0.6);
  \draw[very thick] (3.5,11.2) arc (90:270:1.8);
  \draw[very thick] (3.5,9.2) -- (8.5,9.2);
  \draw[very thick] (3.5,8.8) -- (8.5,8.8);
  \draw[very thick] (3.5,7.6) -- (8.5,7.6);
  \draw [decorate,decoration={brace,amplitude=5}] (8.5,9.4) -- (8.5,7.4);
  \node[rotate=-90] at (9.3,8.4) {\footnotesize$i$};
  \node at (5.5,8.4) {\footnotesize$\vdots$};

  \draw[very thick] (3.5,13.2) -- (0,13.2);
  \draw[very thick] (3.5,12.8) -- (0,12.8);
  \draw[very thick] (3.5,11.6) -- (0,11.6);
  \draw [decorate,decoration={brace,amplitude=5}] (0,11.4) -- (0,13.4);
  \node[rotate=-90] at (-.9,12.4) {\footnotesize$p^{(r)} -i-1$};
  \node at (1.8,12.4) {\footnotesize$\vdots$};
\end{tikzpicture}

%% file: diagrams/jw_2lpk_1_2.tex
\begin{tikzpicture}[scale=0.5]
  \draw (5,1) rectangle (1,5);
  \node at (3,3) {$\JW_{p^{(r)}-1}$};
  \draw[very thick] (1,1.2) -- (0,1.2);
  \draw[very thick] (1,1.6) -- (0,1.6);
  \node at (0.5,3) {$\vdots$};
  \draw[very thick] (1,4.8) -- (0,4.8);
  \draw [decorate,decoration={brace,amplitude=5}] (0,1) -- (0,5);
  \node[rotate=-90] at (-0.9,3) {\footnotesize$p^{(r)}-1$};

  \draw[very thick] (5,4.8) arc (-90:90:0.2);
  \draw[very thick] (5,3.6) arc (-90:90:1.4);
  \draw[very thick] (5,5.2) -- (0,5.2);
  \draw[very thick] (5,6.4) -- (0,6.4);
  \draw [decorate,decoration={brace,amplitude=5}] (0,5) -- (0,6.6);
  \node[rotate=-90] at (-.9,5.8) {\footnotesize$p^{(r)}-1-j$};

  \draw[very thick] (5,1.2) -- (8.5,1.2);
  \draw[very thick] (5,1.6) -- (8.5,1.6);
  \draw[very thick] (5,2.8) -- (8.5,2.8);
  \draw[very thick] (5,3.2) -- (8.5,3.2);
  \draw [decorate,decoration={brace,amplitude=5}] (8.5,3.4) -- (8.5,1);
  \node[rotate=-90] at (9.4,2) {\footnotesize$j$};

  \draw (3.5,13.4) rectangle (7.5,9.4);
  \node at (5.5,11.4) {$\JW_{p^{(r)}-1}$};
  \draw[very thick] (7.5,13.2) -- (8.5,13.2);
  \draw[very thick] (7.5,12.8) -- (8.5,12.8);
  \node at (9,11.4) {$\vdots$};
  \draw[very thick] (7.5,9.6) -- (8.5,9.6);
  \draw [decorate,decoration={brace,amplitude=5}] (8.5,13.4) -- (8.5,9.4);
  \node[rotate=-90] at (9.4,11.4) {\footnotesize$p^{(r)}-1$};

  \draw[very thick] (3.5,9.6 ) arc (90:270:0.2);
  \draw[very thick] (3.5,11.2) arc (90:270:1.8);
  \draw[very thick] (3.5,11.6) arc (90:270:2.2);
  \draw[very thick] (3.5,9.2) -- (8.5,9.2);
  \draw[very thick] (3.5,7.6) -- (8.5,7.6);
  \draw[very thick] (3.5,7.2) -- (8.5,7.2);
  \draw [decorate,decoration={brace,amplitude=5}] (8.5,9.4) -- (8.5,7.0);
  \node[rotate=-90] at (9.4,8.0) {\footnotesize$p^{(r)}-j$};

  \draw[very thick] (3.5,13.2) -- (0,13.2);
  \draw[very thick] (3.5,12.8) -- (0,12.8);
  \draw[very thick] (3.5,12.0) -- (0,12.0);
  \draw[decorate,decoration={brace,amplitude=5}] (0,11.8) -- (0,13.4);
  \node[rotate=-90] at (-.9,12.4) {\footnotesize$j-1$};

  \draw[very thick] (0,9.2) arc (-90:90:0.4);
\end{tikzpicture}

%% file: pjwidempotents.tex
We now turn to extending the work of Burrull, Libedinsky and Sentinelli~\cite{burrull_libedinsky_sentinelli_2019} in generalising the definition of the Jones-Wenzl idempotent to a sensible element of $\TL^k_n$ for all $n$.
The element constructed in~\cite{burrull_libedinsky_sentinelli_2019} is in fact the idempotent defining the projective cover of the trivial module, although it is not explicitly stated as such.
However, the construction present only covers the $\delta = 2$ case where the situation is ``over the integers''.
In this paper we present a generalisation to all $\delta$.
Our results will specialise in the case $\delta = 2$, which corresponds to $\ell = p$.

We briefly recount their methodology here for completeness and introduce terminology to indicate dependence on $\delta$.
The definition that follows is largely a rewrite  of section 2.3 in~\cite{burrull_libedinsky_sentinelli_2019} and the reader is encouraged to peruse that paper for further information.

Recall the definition of $\supp(n)$ in~\cite{spencer_2020_modular}.  If $n + 1 = \sum_{i=a}^b n_i p^{(i)}$ is the $(\ell,p)$-expansion of $n+1$,
\begin{equation}
  \supp(n) = I_n = \{n_bp^{(b)} \pm n_{b-1}p^{(b-1)} \pm \cdots  \pm n_a p^{(a)}-1\}
\end{equation}
We use the notation $I_n$ to keep parity with ~\cite{burrull_libedinsky_sentinelli_2019}.

A number $n$ is called $(\ell, p)$-Adam if $I_n = \{n\}$ so $a=b$.  Equivalently, $n < \ell$, or $n < \ell p$ and is congruent to $-1$ modulo $\ell$ or of the form $c p^{(r)} -1$ for $1 \le c < p$ and $r \ge 1$.
That is to say, by \cref{thm:when_jw}, a number $n$ is $(\ell, p)$-Adam iff $\JW_n$ can be lifted to $\TL^k_n$.

If a number is not $(\ell, p)$-Adam, and $a$ is minimal such that $n_a \neq 0$, define $\f{n} = \sum_{i = a+1}^b n_i p^{(i)} - 1$.
This is known as the $(\ell, p)$-``father'' of $n$.
It is then the case that
\begin{equation}\label{eq:I_n_split}
  I_n = \{i + n_a p^{(a)} \;:\; i \in I_{\f{n}}\} \sqcup \{i - n_a p^{(a)} \;:\; i \in I_{\f{n}}\}.
\end{equation}

We inductively define elements $\pljwnQ \in \TL_n^{\Q(\delta)}$ as
\begin{equation}
  \pljwnQ = \sum_{j \in I_n} \lambda_n^j\; p_n^j \cdot \JW_j \cdot \iota p_n^j
\end{equation}
where $p_n^j$ are elements of $\Hom_{\TLcat}(\underline{n}, \underline{j})$ and $\lambda_n^j \in \Q(\delta)$.
In diagrams,
\begin{equation}
  \vcenter{\hbox{
      \input{diagrams/pljwnQ}
}}
=
\mathlarger{\mathlarger{\sum}}_{j \in I_n}
\lambda_n^j\;
  \vcenter{\hbox{
      \input{diagrams/pn_jw_pn}
}}
\end{equation}

These elements will, in fact, be defined over $\Z[\delta]_{\mathfrak{m}}$ which is to say that neither $m(\delta)$ nor $p$ will divide any denominators in the coefficients of the diagrams.
As such they will descend to $\TL_n^k$ and will be the idempotents of the projective cover of the trivial module.

To define $\pljwnQ$, we induct on the cardinality of $I_n$.
When $I_n = \{n\}$ we set $\pljwnQ = \JW_n$ so that $\lambda^n_n = 1$ and $p^n_n = \id_n$.

Now suppose that $n$ is not $(\ell, p)$-Adam, but all $\lambda_{n'}^j$ and $p_{n'}^j$ are known for $n'$ with smaller cardinality $I_{n'}$.
In particular, $\lambda_{\f{n}}^j$ and $p_{\f{n}}^j$ are all known.
Let $m = n - \f{n} = n_a p^{(a)}$.
Then for each $i \in I_{\f{n}}$, set
\begin{equation}\label{eq:lambda-rec}
  \lambda_n^{i-m} = \frac{[i+1-m]}{[i+1]} \lambda^i_\f{n}\quad,\quad\quad\quad \lambda_n^{i+m} = \lambda_{\f{n}}^i
\end{equation}
and
\begin{equation}\label{eq:recurse_pni}
  \vcenter{\hbox{
    \begin{tikzpicture}[scale=1, transform shape]
      \draw (-.05,-.5) -- (-.9,-.8) -- (-.9, .8) -- (-.05,.5) -- (-.05,-.5);
      \node at (-.45, 0.02) {$p^{i-m}_n$};
      \draw[very thick] (-1.2, 0.6) -- (-.9, 0.6);
      \draw[very thick] (-1.2, -0.6) -- (-.9, -0.6);
      \node at (-1.05, 0.1) {$\vdots$};
      \draw[very thick] (-.05, 0.4) -- (.2, 0.4);
      \draw[very thick] (-.05, -0.4) -- (.2, -0.4);
      \node at (0.1, 0.1) {$\vdots$};
    \end{tikzpicture}
}}
 =
  \vcenter{\hbox{
    \begin{tikzpicture}[scale=1, transform shape]
      \draw (-.15,-.7) -- (-.9,-1) -- (-.9, 1) -- (-.15,.7) -- (-.15,-.7);
      \node at (-.5, 0.02) {$p^{i}_\f{n}$};
      \draw (-.15,-.7) rectangle (.8,.7);
      \node at (.325, 0) {$\JW_i$};

      \draw[very thick] (-1.2, 0.8) -- (-.9, 0.8);
      \draw[very thick] (-1.2, -0.8) -- (-.9, -0.8);
      \node at (-1.05, 0.1) {$\vdots$};

      \draw[very thick] (.8, 0.6) -- (1.85, 0.6);
      \draw[very thick] (.8, 0.1) -- (1.85, 0.1);
      \draw[very thick] (.8, -0.1) -- (1.05, -0.1);
      \draw[very thick] (.8, -0.6) -- (1.05, -0.6);

      \draw[very thick] (1.05, -0.6) arc (90:-90:0.25);
      \draw[very thick] (1.05, -0.1) arc (90:-90:0.75);

      \draw[very thick] (-1.2, -1.1) -- (1.05, -1.1);
      \draw[very thick] (-1.2, -1.6) -- (1.05, -1.6);
      \node at (-1.05, -1.25) {$\vdots$};
      \node at (1.05, -.25) {$\vdots$};
      \node at (1.05, .45) {$\vdots$};
    \end{tikzpicture}
}}
  \quad,\quad\quad\quad
  \vcenter{\hbox{
\begin{tikzpicture}[scale=1, transform shape]
  \draw (-.05,-.5) -- (-.9,-.8) -- (-.9, .8) -- (-.05,.5) -- (-.05,-.5);
  \node at (-.45, 0.02) {$p^{i+m}_n$};
  \draw[very thick] (-1.2, 0.6) -- (-.9, 0.6);
  \draw[very thick] (-1.2, -0.6) -- (-.9, -0.6);
  \node at (-1.05, 0.1) {$\vdots$};
  \draw[very thick] (-.05, 0.4) -- (.2, 0.4);
  \draw[very thick] (-.05, -0.4) -- (.2, -0.4);
  \node at (0.1, 0.1) {$\vdots$};
\end{tikzpicture}
}}
 =
  \vcenter{\hbox{
\begin{tikzpicture}[scale=1, transform shape]
  \draw (-.15,-.5) -- (-.9,-.8) -- (-.9, .8) -- (-.15,.5) -- (-.15,-.5);
  \node at (-.5, 0.02) {$p^{i}_\f{n}$};
  \draw[very thick] (-1.2, 0.6) -- (-.9, 0.6);
  \draw[very thick] (-1.2, -0.6) -- (-.9, -0.6);
  \node at (-1.05, 0.1) {$\vdots$};
  \draw[very thick] (-.15, 0.4) -- (.1, 0.4);
  \draw[very thick] (-.15, -0.4) -- (.1, -0.4);
  \node at (-0, 0.1) {$\vdots$};
  \draw[very thick] (-1.2, -0.9) -- (.1, -0.9);
  \node at (-.5, -1.1) {$\vdots$};
  \draw[very thick] (-1.2, -1.5) -- (.1, -1.5);
\end{tikzpicture}
}}
\end{equation}
This concludes the definition of $\pljwnQ$.

\vspace{1em}
It is clear that in the case $\delta = 2$, which implies $\ell = p$, this coincides\footnote{Readers concerned with the lack of sign in \cref{eq:lambda-rec} should recall that our convention is that loops resolve to $\delta$, instead of $-\delta$.} with the definition given by Burrull, Libedinsky and Sentinelli.
The only changes are to introduce quantum numbers in the fractions in \cref{eq:lambda-rec} and to use our generalized sets $I_n$.

Letting $U_n^i = p_n^i \cdot \JW_i \cdot \iota p_n^j$ so that $\pljwnQ = \sum_{i \in I_n} \lambda^i_n U_n^i$, we have the following proposition.
The proof carries over exactly from~\cite{burrull_libedinsky_sentinelli_2019}, but with the use of quantum numbers in all the fractions.

\begin{proposition}\cite[3.2]{burrull_libedinsky_sentinelli_2019}\label{prop:pljwnQ-idempotent}
  The element $\pljwnQ \in \TL^{\Q(\delta)}_n$ is an idempotent.  Moreover, $\{\lambda^i_n U_n^i\}_{i+1 \in I_n}$ is a set of mutually orthogonal idempotents.
\end{proposition}

The result hinges on the equation (for each $i, j \in I_n$)
\begin{equation}
  \JW_i \cdot \iota p_{n}^i \cdot p_{n}^j \cdot \JW_j =
  \begin{cases}
    \frac{1}{\lambda^i_{n}} \JW_i & i = j\\
    0 & i \neq j
  \end{cases}
\end{equation}
which we may read as $\langle p_n^i,p_n^i \rangle = \frac{1}{\lambda^i_n}$ in $S(n,i)$.

\begin{lemma}
  Each $\TL^{\Q(\delta)}_n \cdot \lambda_n^i U_n^i$ is isomorphic to $S(n,i)$.
\end{lemma}
\begin{proof}
  Recall that we are in the semi-simple case and omit the superscripts.
  Consider the $\TL_n$-morphism $\phi : \TL_n \cdot \lambda_n^i U_n^i \to S(n,i)$ defined by
  \begin{equation}
    a \cdot p^i_n \JW_i \;\iota p^i_n \mapsto a \cdot p^i_n
  \end{equation}
  This is surjective since any nonzero element of $S(n,i)$ generates the entire module and is injective as if $a \cdot p^i_n= 0 \in S(n,i)$, then the morphism $a \cdot p^i_n \in \Hom_\TLcat(\underline{n}, \underline{i})$ factors through some $\underline{j}$ for $j < i$ and hence $a\cdot p^i_n \JW_i = 0$.
\end{proof}

\begin{corollary}\label{cor:char_zero_factors}
  As $\TL^{\Q(\delta)}_n$-modules, $\TL^{\Q(\delta)}_n\cdot \pljwnQ\simeq \bigoplus_{i+1 \in I_n}{S(n, i)}$.
\end{corollary}

Similarly to \cref{prop:pljwnQ-idempotent}, the proof of the following runs identically to that in~\cite{burrull_libedinsky_sentinelli_2019}.
\begin{proposition}\label{prop:absorption}
  The idempotents $\pljwnQ$ satisfy the ``absorption property'',
  \begin{equation}
    \vcenter{\hbox{
      \begin{tikzpicture}[scale=1, transform shape]
        \draw (0,0) rectangle (1.4,1.4);
        \node at (0.7,0.7) {$\pljwnQ$};

        \draw (-1.4,0.4) rectangle (0,1.4);
        \node at (-0.7,0.9) {$\prescript{p}{\ell}{\JW}^{\Q(\delta)}_{\f{n}}$};

        \draw (-1.4,0) rectangle (0,.4);
        \node at (-0.7,0.2) {$\id_m$};
      \end{tikzpicture}
    }}
     =
    \vcenter{\hbox{
      \begin{tikzpicture}[scale=1, transform shape]
        \draw (0,0) rectangle (1.4,1.4);
        \node at (0.7,0.7) {$\pljwnQ$};

        \draw (1.4,0.4) rectangle (2.8,1.4);
        \node at (2.1,0.9) {$\prescript{p}{\ell}{\JW}^{\Q(\delta)}_{\f{n}}$};

        \draw (1.4,0) rectangle (2.8,.4);
        \node at (2.1,0.2) {$\id_m$};
      \end{tikzpicture}
    }}
    =
    \vcenter{\hbox{
      \begin{tikzpicture}[scale=1, transform shape]
        \draw (0,0) rectangle (1.4,1.4);
        \node at (0.7,0.7) {$\pljwnQ$};
      \end{tikzpicture}
    }}
    .
  \end{equation}
\end{proposition}

\Cref{cor:char_zero_factors,prop:absorption}, along with knowledge of the composition factors of the projective cover of the trivial module will give us all the ingredients to show the following:
\begin{proposition}\label{prop:lift}
  The element $\pljwnQ$ can be lifted to $\Z[\delta]_{\mathfrak{m}}$ and therefore when written in the diagram basis, each coefficient can be written as $a / b$ where $b$ does not vanish in $k$.
\end{proposition}
The proof is deferred until \cref{sec:results}.
This allows us to define the $(\ell, p)$-Jones-Wenzl projector on $n$ strands.
\begin{definition}
  Let $(R,\delta)$ be a field with $(\ell, p)$ torsion.
  The $(\ell, p)$-Jones-Wenzl idempotent in $\TL_n^k(\delta)$, denoted $\pljwnR$, is that element obtained by replacing each coefficient $a / b$ in $\pljwnQ$ by its image in $k$.
\end{definition}
To prove \cref{prop:lift}, we will need a further claim that will be shown simultaneously.
\begin{theorem}\label{thm:idem_prim_iso}
  As a left $\TL_n^k$-module,  $\TL_n^k \cdot \pljwnR$ is isomorphic to the projective cover of the trivial $\TL_n^k$-module.
\end{theorem}
As such, we obtain the corollary that the idempotent $\pljwnR$ is primitive.

%% file: diagrams/pljwnQ.tex
\begin{tikzpicture}[scale=1, transform shape]
  \draw (0,0) rectangle (1.4,1.4);
  \node at (.7, .7) {$\pljwnQ$};
  \draw[very thick] (-0.3, 1.2) -- (0, 1.2);
  \node at (-0.15, .8) {$\vdots$};
  \node at (1.55, .8) {$\vdots$};
  \draw[very thick] (-0.3, 0.2) -- (0, 0.2);
  \draw[very thick] (1.7, 1.2) -- (1.4, 1.2);
  \draw[very thick] (1.7, 0.2) -- (1.4, 0.2);
\end{tikzpicture}

%% file: diagrams/pn_jw_pn.tex
\begin{tikzpicture}[scale=1, transform shape]
  \draw (-.4,-.5) rectangle (.4,.5);
  \node at (0, -.05) {$\JW_j$};
  \draw (-.4,-.5) -- (-.9,-.8) -- (-.9, .8) -- (-.4,.5);
  \node at (-.65, 0.02) {$p^j_n$};
  \node at (.65, 0.02) {$p^j_n$};
  \draw (.4,-.5) -- (.9,-.8) -- (.9, .8) -- (.4,.5);
  \draw[very thick] (-1.2, 0.6) -- (-.9, 0.6);
  \node at (-1.05, 0.1) {$\vdots$};
  \draw[very thick] (1.2, 0.6) -- (.9, 0.6);
  \draw[very thick] (-1.2, -0.6) -- (-.9, -0.6);
  \node at (1.05, 0.1) {$\vdots$};
  \draw[very thick] (1.2, -0.6) -- (.9, -0.6);
\end{tikzpicture}

%% file: induction.tex
\subsection{Induction}
In this section, $R$ and $\delta$ will be arbitrary and therefore omitted from the notation.

Recall that $\TL_{n-1} \hookrightarrow \TL_n$ naturally by the addition of a ``through string'' at the lowest sites.
In this way, we may induce $\TL_{n-1}$-modules
\begin{equation}
  M\ind = \TL_n \otimes_{\TL_{n-1}} M
\end{equation}
and restrict $\TL_n$-modules to $\TL_{n-1}$-modules, which we denote $M\restr$.

The proof of the following proposition follows exactly as in the characteristic zero case.
\begin{proposition}\cite[6.3]{ridout_saint_aubin_2014}
  If $\delta \neq 0$ or $(n,m) \neq (2,0)$,
  \begin{equation}
    S(n-1,m)\ind \;\cong\; S(n+1, m)\restr
  \end{equation}
  as $\TL_n$-modules.
\end{proposition}

The power in this is that restriction is easily understood, and again the characteristic zero proof applies over arbitrary rings.
\begin{proposition}\cite[4.1]{ridout_saint_aubin_2014}\label{prop:ses_res}
  There is a short exact sequence of $\TL_{n-1}$ modules,
  \begin{equation}
    0 \to S(n-1, m-1) \to S(n,m)\restr \to S(n-1,m+1) \to 0.
  \end{equation}
\end{proposition}

We will be interested in inducing from $\TL_m$ to $\TL_n$ for $m<n$, which we will denote $M\ind_m^n$ and the corresponding restriction $M\restr_m^n$.
This is achieved by $n-m$ iterations of the induction described above.
A trivial consequence of \cref{prop:ses_res} is that
\begin{corollary}\label{cor:range}
  The module $S(n,m)\ind_n^N$ has a filtration by standard modules, and the multiplicity of $S(N,i)$ is given by
  \begin{equation*}
    \binom{N-n}{(m+N-n-i) / 2}.
  \end{equation*}
  In particular, the only $S(N,i)$ appearing are those for which $m - N + n \le i \le m + N-n$ and for $i \in \{m \pm (N-n)\}$ the factor $S(N,i)$ appears exactly once.
\end{corollary}

\subsection{Projective Covers of the Trivial Module}
Let us suppose that $R$ has $(\ell, p)$-torsion (and that $p = \infty$ if $R$ is characteristic zero and that $\ell = \infty$ if $\delta$ satisfies no quantum number).
Recall that each projective module of $\TL_n$ has a filtration by standard (cell) modules.
For $m \equiv_2 n$, let $P(n,m)$ be the projective cover of the simple head of $S(n,m)$ as a $\TL_n$-module.
A corollary of Theorems 3.4 and 8.4 of \cite{spencer_2020_modular} is
\begin{corollary}
  The multiplicity of $S(n,m)$ in a standard filtration of $P(n,m')$ is 1 iff $m \in I_{m'}$, otherwise it is 0.
\end{corollary}

Let $\nu_{(p)}(x) = 0$ if $\ell \nmid x$ and $\nu_p(x/\ell) + 1$ otherwise, so that $\nu_{(p)}(x)$ gives the position of the least significant nonzero $(\ell, p)$-digit of $x$.
A trivial consequence of the above is that if $S(n,m)$ appears in a standard filtration of $P(n,m')$, then $\nu_{(p)}(m) = \nu_{(p)}(m')$.

Suppose now that $N > n$ and that $m=N-n$ is such that $m < p^{(\nu_{(p)}(n))}$ so $\nu_{(p)}(n) > \nu_{(p)}(m) = \nu_{(p)}(N)$.
Consider the module $P(n,n)\ind_n^N$.
This is a projective module.
If we consider the filtration of $P(n,n)\ind_n^N$ by cell modules, we see by \cref{cor:range} that the $S(N,j)$ appearing in a filtration of $P(n,n)\ind_n^N$ must all satisfy
\begin{equation}
  i-N+n \le j \le i+N-n
\end{equation}
for some $i+1 \in I_n$.
In particular, the trivial module $S(N,N)$ appears exactly once and careful examination of \cref{prop:ses_res} shows that it must appear in the head of $P(n,n)\ind_n^N$.
As such, there is a \emph{unique} summand of $P(n,n)\ind_n^N$ isomorphic to $P(N,N)$.
Additionally, from \cref{eq:I_n_split,cor:range}, where $n_ap^{(a)} = m = N-n$, we have further that every module $S(N,i)$ appearing in a filtration of both $P(n,n)\ind_n^N$ and $P(N,N)$ does so in each exactly once.

\vspace{1em}
To make these observations numerical, consider the Grothendieck group of $\TL_n$, denoted here by $G_0(\TL_n)$.
This is a free abelian group with basis the isomorphism types of simple $\TL_n$ modules.
That is to say, its elements are given by (isomorphism classes of) finite-dimensional $\TL_n$-modules, modulo the relation that if there is a short exact sequence $0 \to M_1 \to M_2 \to M_3 \to 0$, then $[M_2] = [M_1] + [M_3]$, where $[M]$ is the class of $\TL_n$-module in $G_0(\TL_n)$.
Similarly, we will consider $K_0(\TL_n)$, the free abelian group on (isomorphism classes of) indecomposable projective modules with the same relation.

A consequence of~\cite[3.6]{graham_lehrer_1996} is that ${\left\{[S(n,i)]\right\}}_{i \in \Lambda_0}$ form a basis for $G_0(\TL_n)$.
In $G_0(\TL_n)$, the above discussion can be rephrased as
\begin{equation}\label{eq:grot_ind}
  \left[P(n,n)\ind_n^N\right] = \left[P(N,N)\right] + \sum_{j\in J} \left[S(N,j)\right]
\end{equation}
where $J$ is a multiset disjoint from $I_n$.
\Cref{eq:grot_ind} should be read as a rephrasing of~\cite[Lemma 4.11]{burrull_libedinsky_sentinelli_2019}, but without recourse to the $p$-Kazhdan-Lusztig basis, Soergel bimodule theory or Schur-Weyl duality.

%% file: results.tex
Recall the $cde$-Triangle of~\cite[\S 9.5]{webb_2016}:
\begin{center}
\begin{tikzcd}
  &G_0(\TL_n^{F})
  \arrow[dr,"d"]& \\
  K_0(\TL_n^k)\arrow[ur,"e"] \arrow[rr,hookrightarrow,"c"]&& G_0(\TL_n^k)
\end{tikzcd}
\end{center}
Here, $c$ simply takes the image of a module $[M] \in K_0(\TL_n^k)$ to its image $[M] \in G_0(\TL_n^k)$.

The map $d$ is defined on the basis of $G_0(\TL_n^{F})$.
Recall that this is a semi-simple algebra and the simple modules are exactly $S(n,i)$ for $i \le n$ and $i\equiv_2 n$.
We define $d([S(n,i)]) = [S(n,i)]$.
It is clear that $d$ is the transpose of the decomposition matrix.

To define $e$, one uses idempotent lifting techniques.
Let $P$ be a projective module of $\TL_n^k$ and suppose that $P \simeq \TL_n^k \cdot e$ for some idempotent $e$ defined in diagrams over $k$.
Then lift $e$ to an idempotent over $F$.
This defines a projective $\TL_n^F$-module $\hat{P}$ and the image $e[P]$ is $[\hat{P}]$.

It is classical theory that $c = de$.

\vspace{1em}

We now have all the results required to prove \Cref{prop:lift,thm:idem_prim_iso}.

\begin{proof}
  We show the results by mathematical induction on $|I_n|$.
  When $|I_n| = 1$, we have that $\pljwnR = \JW_n$ exists in $\TL_n^k$ by \cref{thm:when_jw}.

  Otherwise, assume both \Cref{prop:lift,thm:idem_prim_iso} hold for all $n'$ with $|I_{n'}| < |I_n|$.
  In particular the results are known for $\f{n}$.
  Thus $\pljwfnQ$ descends to an idempotent $\pljwfnR$ in $\TL_n^k$ which  describes the projective module $P(\f{n},\f{n})$.

  Now consider the module $P(\f{n}, \f{n})\ind_{\f{n}}^n$.
  This is isomorphic to the module $\TL_n^k \cdot g$ where $g$ is the idempotent $\pljwfnR \otimes \id_{n-\f{n}}$ in $\TL_n^k$.
  The composition factors of this module are described in \cref{eq:grot_ind} and the discussion preceding it.
  In particular, if $e$ is the unique idempotent of $\TL_n^k$ describing the projective cover of the trivial module, then
  $e \cdot P(\f{n}, \f{n})\ind_{\f{n}}^n$ is a non-zero module isomorphic to $P(n,n)$.

  As such $e \cdot g = g\cdot e = e$.
  If we lift both of these elements to $F$, say to $\hat{e}$ and $\hat{g}$ we obtain $\hat{e} \cdot \hat{g} = \hat{g} \cdot \hat{e} = \hat{e}$, so in particular
  \begin{equation}
    \hat{e} \in \hat{g} \cdot \TL_n^F \cdot \hat{g} \cong \End_{\TL^F_n}(\TL_n^F \cdot \hat{g}).
  \end{equation}

  But $\TL_n^F \cdot \hat{g}$ is exactly the lift of $P(\f{n},\f{n})\ind_{\f{n}}^n$ to characteristic zero and so $\hat{g} = \pljwfnQ \otimes \id_{n-\f{n}}$.
  Thus \cref{prop:absorption} claims that $\pljwnQ \cdot \hat{g} = \hat{g} \cdot \pljwnQ = \pljwnQ$.
  Thus too
  \begin{equation}
    \pljwnQ \in \hat{g} \cdot \TL_n^F \cdot \hat{g} \cong \End_{\TL^F_n}(\TL_n^F \cdot \hat{g}).
  \end{equation}

  However, recall that $\TL_n^F$ is semi-simple and that $e [P(\f{n},\f{n})\ind_{\f{n}}^n] = e[P(n,n)] + M$ where $M$ does not have support intersecting that of $e[P(n,n)]$.
  This is to say that there is a unique idempotent in $\End_{\TL_n^F}\left(P(\f{n},\f{n})\ind_\f{n}^n\right)$ with image having composition factors given by $e[P(n,n)]$.
  Clearly by construction $\hat{e}$ is such an idempotent, and \cref{cor:char_zero_factors} shows that $\pljwnQ$ is too.
  Hence they must be equal.
\end{proof}

%% file: traces.tex
The trace map $\tau : \Hom_{\TLcat}(\underline{n}, \underline{m}) \to \Hom_{\TLcat}(\underline{n-1}, \underline{m-1})$ is defined diagrammatically as
\begin{equation}
  \tau f = 
  \vcenter{\hbox{
  \begin{tikzpicture}[scale=1, transform shape]
    \draw (0,0) -- (1,.5) -- (1,1.5) -- (0,2) -- (0,0);
    \node at (0.5,1) {$f$};
    \draw[very thick] (0,0.2)--(-.1,0.2);
    \draw[very thick] (0,0.4)--(-.5,0.4);
    \draw[very thick] (0,1.8)--(-.5,1.8);
    \draw[very thick] (-.1,.2) arc (90:270:0.2);
    \draw[very thick] (-.1,-.2) -- (1.1,-.2);
    \draw[very thick] (1,.7) -- (1.1,.7);
    \draw[very thick] (1.1,-.2) arc (-90:90:.45);
    \draw[very thick] (1,.8) -- (1.8,.8);
    \draw[very thick] (1,1.4) -- (1.8,1.4);
    \node at (-.25,1.2) {$\vdots$};
    \node at (1.4,1.2) {$\vdots$};
  \end{tikzpicture}
  }}
\end{equation}
By linear extension, this induces the trace $\tau : \TL_n \to \TL_{n-1}$ for each $n \ge 1$ as well as the ``full trace'' $\tau^n : \TL_n \to TL_{0}\simeq k$.
These full traces were critical in the construction of the Jones polynomial invariant of knots.
Computations of partial traces of generalized Jones-Wenzl idempotents are necessary for computations of certain truncated categories~\cite{sutton_tubbenhauer_2021, spencer_2021}.
We can also compare the partial trace of generalized and traditional Jones-Wenzl idempotents to get a sense of how the mixed characteristic generalises the semi-simple case.

The action of the trace on Jones-Wenzl elements over characteristic zero is well understood and is almost folklore.
\begin{lemma}\cite[Lemma 3.1]{burrull_libedinsky_sentinelli_2019}\label{lem:trace_jw}
  In $\TL_n^{\Q(\delta)}$,
  \begin{equation}\label{eq:trace_jw}
    \tau^m \left( \JW_n \right) = \frac{[n+1]}{[n+1-m]}\, \JW_{n-m}.
  \end{equation}
\end{lemma}
However, this breaks down over positive characteristic.
For example, the form of $\JW_{2p^{(r)}-1}$ given in \cref{prop:construct_2} makes clear that 
\begin{equation}\label{eq:trace_2}
  \tau\left(\JW_{2p^{(r)}-1}\right) = 2\; J_{p^{(r)}-1} \otimes J_{-p^{(r)}+1}.
\end{equation}
This can be read as stating that $\frac{[2p^{(r)}]}{[2p^{(r)}-1]}\JW_{2p^{(r)}-2}$ exists over $k$ and has value given by \cref{eq:trace_2}.
Notice that such a morphism has zero through-degree.

Recall from~\cite[Lemma 2.2]{spencer_2020_modular} that $[\ell m]/ [\ell n] = m / n$ in $k$.  The following is a trivial corollary of \cref{lem:trace_jw}.
\begin{corollary}\label{cor:trace_proper}
  Let $1 \le a \le p$ and $r \ge 1$.  Then as elements in $\TL^k$,
  \begin{equation}
    \tau^{p}\left(\JW^k_{a p^{(r)}-1}\right) = \frac{a}{a-1}\; \JW^k_{(a-1)p^{(r)}-1}.
  \end{equation}
\end{corollary}

We now ask how the trace acts on elements $\pljwnQ$.

\begin{proposition}
  Recall that $n+1 = \sum_{i=a}^b n_i p^{(i)}$ and let $t = p^{(a)}$.
  \begin{equation}
    \tau^{t}\left(\pljwnQ\right)=
    \begin{cases}
      \frac{[n_a]}{[n_a-1]}\pljwnmtQ & n_a > 1\text{ and }a=0\\
      \frac{n_a}{n_a-1}\pljwnmtQ & n_a > 1\text{ and }a>0\\
      [2]\pljwnmtQ& n_a = 1\text{ and }a=0\\
      2\pljwnmtQ& n_a = 1\text{ and }a>0
    \end{cases}
  \end{equation}
\end{proposition}
\begin{proof}
Let $m = n- \f{n} = n_ap^{(a)} > p^{(a)} = t$ so that in particular $\f{n} = \f{n-t}$ and 
\begin{equation}
  I_{n} = \left(I_\f{n} + m\right) \sqcup \left(I_{\f{n}} - m\right)
  \quad;\quad
  I_{n-t} = \left(I_\f{n} + m - t\right) \sqcup \left(I_{\f{n}} - m + t\right)
\end{equation}
Then notice that for $i \in I_\f{n}$,
\begin{align}
  \tau^t \; U_{n}^{i-m} &= U_{n-t}^{i-m+t}\label{eq:trace_u_1}\\
  \tau^t \; U_{n}^{i+m} &= \frac{[i+m+1]}{[i+m-t+1]}U_{n-1}^{i+m-t}\label{eq:trace_u_2}
\end{align}
Recall from \cref{eq:lambda-rec} that
\begin{equation}\label{eq:long_form_pljwnQ}
  \pljwnQ = \sum_{i+1 \in I_\f{n}} \left(
    \frac{[i+1-m]}{[i+1]} \lambda_\f{n}^i U_n^{i-m} + 
    \lambda_\f{n}^i U_n^{i+m}
  \right)
\end{equation}
and also
\begin{equation}\label{eq:long_form_pljwnmoQ}
  \pljwnmtQ = \sum_{i+1 \in I_\f{n}} \left(
    \frac{[i+1-m+t]}{[i+1]} \lambda_\f{n}^i U_{n-t}^{i-m+t} + 
    \lambda_\f{n}^i U_{n-t}^{i+m-t}
  \right)
\end{equation}
Then \cref{eq:trace_u_1,eq:trace_u_2} with \cref{eq:long_form_pljwnQ} give
\begin{multline}
  \tau^t\left( \pljwnQ \right) = \sum_{i+1 \in I_\f{n}} \Bigg(
    \frac{[i+1-m]}{[i+1-(m-t)]}
    \frac{[i+1+t-m]}{[i+1]} \lambda_\f{n}^i U_{n-t}^{i-m+t}\\ + 
    \frac{[i+1+m]}{[i+1+(m-t)]}
    \lambda_\f{n}^i U_{n-t}^{i+m-t}
  \Bigg)
\end{multline}
Now, for all $i+1 \in I_\f{n}$ we have that $\ell \mid i+1$ and so $[i+1 + m] = \pm [m]$ (depending on if $[i+1] = \pm 1$) and similarly for $[i+1-m]$.
Hence we see that all the factors actually fall out and
\begin{equation}\label{eq:t-form-trace-1}
  \tau^t\left(\pljwnR \right) = \frac{[m]}{[m-t]} \pljwnmtR,
\end{equation}
recovering a form of \cref{eq:trace_jw} for $(\ell, p)$-Jones-Wenzl elements.
However, if $a>0$ then this simplifies to $n_a/(n_a-1)$ and if $a=0$ this is simply $[n_a]/[n_a-1]$ as desired.

If, on the other hand, $\f{n} = n - t$ so that $m = t= p^{(a)}$, we get a slightly different beahvior (indeed, this is necessary to avoid a division by zero in \cref{eq:t-form-trace-1}).
Recall
\begin{equation}
  [i][j-k] + [j][k-i] + [k][i-j] = 0,
\end{equation}
so that
\begin{equation}
  [a + b] + [a-b] = \frac{[2b][a]}{[b]}.
\end{equation}
In particular
\begin{equation}
  [i+1 + p^{(a)}] + [i+1-p^{(a)}] = \frac{[2p^{(a)}][i+1]}{[p^{(a)}]} =
  \begin{cases}
    [2][i+1] & a = 0\\
    2[i+1] & a > 0
  \end{cases}.
\end{equation}
Thus if $n$ is $(\ell,p)$-Adam, then
\begin{align}
\nonumber  \tau^t\left( \pljwnR \right) &= 
  \sum_{i+1 \in I_\f{n}} \Bigg(
    \frac{[i+1-t]}{[i+1]}
    \lambda_\f{n}^i U_{n-1}^{i} +
    \frac{[i+1+t]}{[i+1]}
    \lambda_\f{n}^i U_{n-1}^{i}
  \Bigg)\\
\nonumber  &=
  [2] \sum_{i+1 \in I_\f{n}}
  \lambda_\f{n}^i U_\f{n}^i\\
  &=
  \begin{cases}
    [2]\pljwnmtR & a = 0\\
    2\pljwnmtR & a > 0
  \end{cases}.
\end{align}
\end{proof}

A result of the above computation is that $\tau^m \left(\pljwnR\right) = [2][n_0] \pljwfnR$ if $a = 0$ and $2n_a\pljwfnR$ otherwise.

The question of tracing $\pljwfnR$ by amounts other than valid $t$
is still not well understood.
For example, if we try trace a single strand when $n$ is Adam,
is clear that $\pljwnR$ is the image of $\JW_n$ and hence $\tau\pljwnR$ is the image of $[n+1]/[n]\,\JW_n$.
As such, since $\JW_n$ is killed by the action of all $u_i$, so too must $\tau\pljwnR$ be.
By considering all diagrams of the form $|x\rangle\langle y|$ (see \cite{spencer_2020_modular} for notation) where $y$ is a fixed diagram of maximal degree $d$, we see that this implies the existence of a trivial submodule of $S(n,d)$.
Knowledge of where these modules could exist then give us restrictions on the valid values of $d$.

As a first attempt at understanding the trace we may simply ask if $\tau\pljwnR$ has maximal through degree.
That is, does it have a non-zero coefficient of the identity diagram?
Clearly this is always the case if $n_0\neq 0$.
If $n$ is $(\ell,p)$-Adam, then \cref{lem:trace_jw} holds and $\tau^m(\pljwnR)$ has maximal through degree iff $\nu_{(p)}(n+1) = \nu_{(p)}(n+1-m)$.  Note that it is guaranteed that $\nu_{(p)}(n+1) \ge \nu_{(p)}(n+1-m)$ as $n$ is Adam.

We can use the definition of $\pljwnQ$ as a sum of the $U_n^i$ to calculate the case when $n$ has nonzero $(\ell,p)$ valuation and we trace by an amount less than $p^{(a)}$ (as any larger can be covered by \cref{cor:trace_proper} first).
\begin{proposition}
  Let $n+1 = \sum_{i=a}^b n_i p^{(i)}$ for $b>a>0$ and $t < n_a p^{(a)}$.
  Then the coefficient of the identity diagram in $\tau^t\left( \pljwnQ\right)$ is $[2]^t$.
\end{proposition}
\begin{proof}
  We use the definition $\pljwnQ = \sum_{i+1\in I_n}\lambda_n^i U_n^i$ and evaluate $\tau^t U_n^i$.
  Since $n$ is not $(\ell, p)$-Adam, we can split $I_n$ into $I_\f{n} + m$ and $I_\f{n}-m$ for $m = n - \f{n}> t$.

  Referring to \cref{eq:recurse_pni}, we see that for $i \in I_{\f{n}}$, $\tau^t U_n^{i+m}$ has the unit diagram with nonzero coefficient iff $i = \f{n}$.  The coefficient is $[2]^t$ as $\lambda_n^{\f{n}} = 1$.
  On the other hand, the through degree of $U_n^{i-m}$ is $i-m+t < n-m+t<n$.
\end{proof}